\documentclass[11pt]{article}

\usepackage[margin=1in]{geometry}
\usepackage{authblk}
\usepackage{amsmath,amssymb,amsthm}
\usepackage{mathtools}
\usepackage{bm}
\usepackage{algorithm}
\usepackage{algorithmic}
\usepackage{booktabs}
\usepackage{multirow}
\usepackage{array}
\usepackage{natbib}
\usepackage{url}
\usepackage{graphicx}
\usepackage{textcomp}
\usepackage{hyperref}

\newcommand{\R}{\mathbb{R}}
\newcommand{\E}{\mathbb{E}}
\newcommand{\Prob}{\mathbb{P}}
\newcommand{\cU}{\mathcal{U}}
\newcommand{\cX}{\mathcal{X}}
\newcommand{\cG}{\mathcal{G}}

\newcommand{\cB}{\mathcal{B}}
\newcommand{\cZ}{\mathcal{Z}}
\newcommand{\cD}{\mathcal{D}}
\newcommand{\cM}{\mathcal{M}}
\newcommand{\ip}[2]{\langle #1, #2 \rangle}
\newcommand{\norm}[1]{\|#1\|}

\newcommand{\btheta}{\bm{\theta}}
\newcommand{\bmu}{\bm{\mu}}
\newcommand{\bu}{\bm{u}}
\newcommand{\bw}{\bm{w}}
\newcommand{\bx}{\bm{x}}
\newcommand{\bxi}{\bm{\xi}}
\newcommand{\bL}{\bm{L}}
\newcommand{\bg}{\bm{g}}

\DeclareMathOperator{\argmin}{argmin}
\DeclareMathOperator{\argmax}{argmax}
\DeclareMathOperator{\ri}{ri}
\DeclareMathOperator{\tr}{tr}
\DeclareMathOperator{\diag}{diag}
\DeclareMathOperator{\chol}{chol}

\DeclareMathOperator{\conv}{conv}

\newcommand{\clsubd}{\partial^{\mathrm{C}}}
\newcommand{\ipF}[2]{\langle #1, #2 \rangle_F}

\newtheorem{theorem}{Theorem}
\newtheorem{lemma}[theorem]{Lemma}
\newtheorem{proposition}[theorem]{Proposition}
\newtheorem{corollary}[theorem]{Corollary}

\theoremstyle{definition}
\newtheorem{definition}{Definition}
\newtheorem{assumption}{Assumption}
\newtheorem{remark}{Remark}

\begin{document}

\title{End-to-End Learning of Correlated Operating Reserve Requirements in Security-Constrained Economic Dispatch}

\author[1]{Owen Shen\thanks{Corresponding author. Email: \href{mailto:owenshen@mit.edu}{owenshen@mit.edu}.}}
\author[2]{Hung-po Chao}
\author[1]{Haihao Lu}
\author[1]{Patrick Jaillet}

\affil[1]{Massachusetts Institute of Technology}
\affil[2]{Energy Trading Analytics, LLC}

\date{}

\maketitle

\begin{abstract}
Operating reserve requirements in security-constrained economic dispatch (SCED) depend strongly on the assumed correlation structure of renewable forecast errors, yet that structure is usually specified exogenously rather than learned for the dispatch task itself. This paper formulates correlated reserve-set design as an end-to-end trainable robust optimization problem: choose the ellipsoidal uncertainty-set shape to minimize robust dispatch cost subject to a target coverage requirement. By profiling the coverage constraint into a shape-dependent radius, the original bilevel problem becomes a single-stage differentiable objective, and KKT/dual information from the SCED solve provides task gradients without differentiating through the solver. For unknown distributions, a four-way train/tune/calibrate/test split combines a smoothed quantile-sensitivity estimator for training with split conformal calibration for deployment, yielding finite-sample marginal coverage under exchangeability and a consistent gradient estimator for the smoothed objective. The same task gradient can also be passed upstream to context-dependent encoders, which we report as a secondary extension. The framework is evaluated on the IEEE~118-bus system with a coupled SCED formulation that includes inter-zone transfer constraints. The learned static ellipsoid reduces dispatch cost by about 4.8\% relative to the Sample Covariance baseline while maintaining empirical coverage above the target level.
\end{abstract}


\section{Introduction}
\label{sec:introduction}

Electricity markets rely on operating reserve procurement and pricing mechanisms to manage uncertainty arising from high penetrations of wind and solar generation. Renewable forecast errors exhibit strong spatial and temporal correlations due to common weather patterns~\citep{hodge2011wind,pinson2013wind}. When such correlations are ignored, reserve capacity may be inefficiently allocated, leading to excessive costs and distorted price signals. A central modeling question is therefore how correlated reserve requirements should be parameterized for the dispatch task itself.

Robust optimization provides a principled framework for incorporating uncertainty into power system operations~\citep{ben2009robust,bertsimas2011theory}. In security-constrained economic dispatch (SCED), uncertain parameters are assumed to lie within a prescribed \emph{uncertainty set}, and the dispatch must remain feasible for all realizations within this set. The geometry of the uncertainty set fundamentally determines the conservatism and cost of the solution~\citep{bertsimas2013adaptive}.

The resulting design problem is inherently coupled: the uncertainty-set geometry affects reserve requirements, the SCED solve determines the economic value of that geometry, and the reliability requirement determines the radius needed for coverage. The goal in this paper is to learn that uncertainty-set geometry that is economically aligned with downstream dispatch. We therefore treat reserve-set design as an end-to-end trainable problem: profile the coverage requirement into a shape-dependent radius, optimize the resulting objective with KKT/dual information from the dispatch solve, and obtain a task gradient that can also be passed upstream to contextual models.

\subsection{Literature Review}

The most relevant prior work falls into six strands: robust power-system optimization, data-driven calibration, decision-focused learning, learned uncertainty sets, conformal prediction, and renewable forecast modeling.

\emph{A) Robust Optimization for Power Systems.}
Robust methods have been widely adopted for unit commitment and economic dispatch. Bertsimas et al.~\citep{bertsimas2013adaptive} developed adaptive robust optimization for security-constrained unit commitment with polyhedral uncertainty sets. Jiang et al.~\citep{jiang2012robust} proposed two-stage robust unit commitment, and Zeng and Zhao~\citep{zeng2013solving} introduced column-and-constraint generation. Lorca et al.~\citep{lorca2016multistage} extended these to multistage settings, while Jabr~\citep{jabr2013adjustable} developed adjustable robust OPF with ellipsoidal sets. These approaches assume a fixed uncertainty set geometry determined a priori.

\emph{B) Data-Driven and Distributionally Robust Approaches.}
Bertsimas et al.~\citep{bertsimas2018datadriven} proposed data-driven robust optimization using hypothesis testing to calibrate set size. Distributionally robust optimization (DRO) optimizes over ambiguity sets: Mohajerin Esfahani and Kuhn~\citep{esfahani2018data} developed Wasserstein DRO with tractable reformulations, Gao and Kleywegt~\citep{gao2023distributionally} established finite-sample guarantees, and Van Parys et al.~\citep{vanparys2021data} proved asymptotic optimality. In power systems, Xiong et al.~\citep{xiong2017distributionally} applied moment-based DRO to unit commitment. Roald et al.~\citep{roald2023power} survey optimization under uncertainty in power systems. These methods do not directly optimize uncertainty set geometry for downstream costs.

\emph{C) Decision-Focused Learning and Differentiable Optimization.}
Decision-focused learning trains models to minimize downstream decision cost~\citep{elmachtoub2022smart,donti2017task}. Differentiable optimization layers~\citep{amos2017optnet,agrawal2019differentiable,bolte2021nonsmooth} enable backpropagation through convex programs. Our envelope-based approach avoids differentiating through solvers entirely.

\emph{D) Learned Uncertainty Sets.}
Wang et al.~\citep{wang2023learning} proposed learning decision-focused uncertainty sets via implicit differentiation. Chenreddy and Delage~\citep{chenreddy2024end} developed end-to-end conditional robust optimization. Goerigk and Kurtz~\citep{goerigk2023data} used neural networks to predict uncertainty set parameters. What remains missing for reserve procurement is an application-driven formulation in which the uncertainty-set geometry, the SCED solve, and the coverage calibration are optimized as one pipeline. The framework in this paper closes that gap by profiling the coverage constraint and using KKT/dual sensitivities from SCED to optimize the resulting single-stage objective without solver backpropagation.

\emph{E) Conformal Prediction.}
Conformal prediction~\citep{vovk2005algorithmic} provides distribution-free uncertainty quantification under exchangeability. Romano et al.~\citep{romano2019conformalized} developed conformalized quantile regression, and Angelopoulos and Bates~\citep{angelopoulos2023conformal} provide a comprehensive tutorial. Johnstone and Cox~\citep{johnstone2021conformal} connected conformal prediction to robust optimization via Mahalanobis distance. Here, conformal calibration plays a specific operational role: once the shape is learned, it calibrates the radius needed to meet the reserve-coverage target.

\emph{F) Renewable Forecast Uncertainty.}
Wind and solar forecast errors exhibit complex spatial and temporal correlations~\citep{hodge2011wind,pinson2013wind}. Hong and Fan~\citep{hong2016probabilistic} surveyed probabilistic load forecasting. These works motivate the need for correlation-aware uncertainty sets that adapt to varying forecast error patterns---precisely what the learned Cholesky parameterization is designed to capture.

Taken together, these literatures motivate an end-to-end reserve-learning formulation in which uncertainty-set shape is learned for the dispatch task itself while coverage is enforced out of sample. That is the role of the framework developed below.

\subsection{Contribution and Organization}

\paragraph{Main contribution.}
\begin{itemize}
\item \textbf{Modeling and reformulation.} We formulate correlated reserve-set design in SCED as an end-to-end trainable robust optimization problem, where the uncertainty-set shape is optimized against the downstream robust dispatch value function and the coverage constraint is profiled into a shape-dependent radius. This yields a single-stage gradient-friendly objective, and the paper establishes the corresponding profiling and envelope-gradient results that justify using SCED dual information rather than solver differentiation.

\item \textbf{Algorithm and calibration.} We develop a practical learning procedure based on a train/ tune/ calibrate/ test split, where the tuning set supports estimation of the smoothed profiled gradient and the calibration set supplies a conformal radius for the final deployed set. The paper further gives a finite-sample conformal coverage guarantee under exchangeability and a consistency result for the smoothed gradient estimator under standard regularity conditions.

\item \textbf{Empirical study.} We validate the framework on the IEEE~118-bus system and show that the learned static ellipsoid reduces dispatch cost by 4.8\% relative to the Sample Covariance baseline while maintaining empirical coverage above the target level. We further study the method under coupled SCED with transfer constraints, showing that the same formulation and gradient machinery remain effective in a more operationally constrained setting.
\end{itemize}

\paragraph{Organization.}
Section~\ref{sec:formulation} fixes the uncertainty score, the robust dispatch value function, and the coverage-constrained learning problem. Section~\ref{sec:methodology} develops the reformulation in three steps: dual-based sensitivity of $V$, profiled gradients through the shape-dependent radius, and final conformal calibration. Section~\ref{sec:algorithm} turns these ingredients into a training pipeline and then reports a secondary contextual extension. Section~\ref{sec:application} presents the coupled SCED study, while the simpler decoupled zonal benchmark and the target-level sweep are deferred to Appendix~\ref{app:zonal-reserve}.

\section{End-to-End Reserve-Learning Problem}
\label{sec:formulation}

In electricity markets, operating reserve procurement must hedge against \emph{net load uncertainty}---the aggregate forecast error arising from renewable generation variability and load prediction errors. Let $U \in \R^d$ denote the uncertainty realization, where $d$ corresponds to the number of uncertainty sources (e.g., wind and solar regions, load zones). The application problem addressed here is to learn an uncertainty set for $U$ that is economical for SCED while still meeting a prescribed coverage level.

This section fixes three objects that will be used throughout the paper: a shape-dependent uncertainty score, the robust dispatch value function, and the coverage-constrained learning problem that will later be reduced to a shape-only objective. The central design choice is the \emph{uncertainty set} $\cU$ specifying which realizations must be hedged. This set must be large enough to cover most realizations (reliability) yet small enough to avoid excessive reserve costs (efficiency). We parameterize uncertainty sets via two components:
\begin{itemize}
\item The \textbf{shape} parameter $\btheta$ captures correlation structure---e.g., if wind forecast errors in neighboring regions are positively correlated, the uncertainty set elongates in that direction.
\item The \textbf{size} parameter $\rho$ controls overall conservatism---larger $\rho$ means more coverage but higher reserve costs.
\end{itemize}
The shape is the trainable object; the size will later be calibrated to meet the target coverage level. This separation is what allows the original coverage-constrained problem to be reduced to a single-stage differentiable objective.

\subsection{Score-Based Uncertainty Sets and Gauge Interpretation}

Fix an uncertainty dimension $d \geq 1$ and a parameter set $\Theta \subseteq \R^p$. We begin with a shape-dependent score $s_{\btheta}(\bu)$ that measures how far a realization $\bu$ lies from the center of the uncertainty set. For convex unit sets, this score is the \emph{gauge} (or Minkowski functional); in the ellipsoidal case used later, it reduces to a whitened Euclidean norm.

\begin{definition}[Parameterized Uncertainty Set]
For shape $\btheta \in \Theta$ and radius $\rho > 0$,
\begin{equation}\label{eq:uset}
\cU_{\btheta,\rho} := \{\bu \in \R^d : s_{\btheta}(\bu) \leq \rho\}
\end{equation}
where $s_{\btheta}: \R^d \to [0,\infty)$ is the gauge function of the unit set $\cU_{\btheta,1}$.
\end{definition}

The gauge function generalizes the notion of ``distance from the origin'' to non-spherical sets. For ellipsoidal uncertainty sets---the primary focus of this paper---$\btheta$ is the Cholesky factor $\bL$ of the covariance matrix $\Sigma = \bL\bL^\top$, and the gauge $s_{\bL}(\bu) = \|\bL^{-1}\bu\|_2$ measures how many ``standard deviations'' $\bu$ lie from the origin in the whitened coordinate system. When forecast errors in different zones are positively correlated, $\Sigma$ has off-diagonal structure, and the ellipsoid elongates along the correlated direction.

\begin{assumption}[Gauge Regularity]\label{asm:gauge}
For each $\btheta \in \Theta$, $\cU_{\btheta,1}$ is nonempty, convex, compact, and contains $\bm{0}$ in its interior. Consequently, $\cU_{\btheta,\rho} = \rho \cdot \cU_{\btheta,1}$ for all $\rho > 0$.
\end{assumption}

The support function converts uncertainty-set geometry into worst-case directional exposure, which is why it appears directly in the robust constraints below.

\begin{definition}[Support Function]
$\sigma_C(\bw) := \sup_{\bu \in C} \ip{\bw}{\bu}$.
\end{definition}

Operationally, $\sigma_C(\bw)$ is the worst-case effect of uncertainty in direction $\bw$. By homogeneity, $\sigma_{\cU_{\btheta,\rho}}(\bw) = \rho \cdot \sigma_{\cU_{\btheta,1}}(\bw)$ (Lemma~\ref{lem:support-scaling} in Appendix~\ref{app:supporting}).

\subsection{Robust Optimization Value Function}

Given a shape--radius pair $(\btheta,\rho)$, the downstream task is the robust dispatch cost. We capture that task abstractly by the value function $V(\btheta,\rho)$, which will be referenced in every subsequent learning formulation. Let $\cX \subseteq \R^{n_x}$ be a closed convex decision set, $f: \R^{n_x} \to \R \cup \{+\infty\}$ a convex objective, and $a_j$ convex constraint functions. Fix exposure vectors $\bw_1, \ldots, \bw_m \in \R^d$ and scalars $b_1, \ldots, b_m$.

\begin{equation}\label{eq:robust}
\begin{aligned}
V(\btheta,\rho) := \inf_{\bx \in \cX} \quad & f(\bx) \\
\text{s.t.} \quad & a_j(\bx) + \sigma_{\cU_{\btheta,\rho}}(\bw_j) \leq b_j, \quad j = 1,\ldots,m,
\end{aligned}
\end{equation}
with $V(\btheta,\rho)=+\infty$ if the feasible set is empty.

\begin{assumption}[Local Slater and Dual Attainment]\label{asm:slater}
Fix $(\btheta,\rho)$ with $V(\btheta,\rho)<+\infty$. Assume $f$ is bounded below on $\cX$ (i.e., $\inf_{\bx\in\cX}f(\bx)>-\infty$) and there exists $\bar{\bx}\in\ri(\cX)$ with $f(\bar{\bx})<+\infty$ satisfying all constraints strictly.
\end{assumption}

Under Assumption~\ref{asm:slater}, strong duality holds and the set of optimal dual multipliers $\cM^*(\btheta,\rho):=\argmax_{\bmu\ge 0}q(\bmu;\btheta,\rho)$ is nonempty (Lemma~\ref{lem:strong-duality}).

\subsection{The Bilevel Learning Problem}

Using the robust dispatch value function $V(\btheta,\rho)$ from~\eqref{eq:robust}, we consider the following coverage-constrained end-to-end learning problem. Let $U \sim P$ be a random uncertainty realization. Define the gauge score CDF $F_{\btheta}(r) := \Prob(s_{\btheta}(U) \leq r)$ and fix target coverage $\tau \in (0,1)$.

\begin{equation}\label{eq:bilevel}
\boxed{
\min_{\btheta \in \Theta,\, \rho > 0} V(\btheta,\rho) \quad \text{s.t.} \quad \Prob(U \in \cU_{\btheta,\rho}) \geq \tau
}
\end{equation}

Define the $\tau$-quantile radius $\rho_P(\btheta) := \inf\{r > 0 : F_{\btheta}(r) \geq \tau\}$.

\begin{proposition}[Radius Profiling]\label{prop:profiling}
Under Assumption~\ref{asm:gauge}, the bilevel problem~\eqref{eq:bilevel} reduces to
\begin{equation}\label{eq:profiled}
\min_{\btheta \in \Theta} J_P(\btheta) := V(\btheta, \rho_P(\btheta)).
\end{equation}
\end{proposition}

\begin{proof}
The map $\rho \mapsto V(\btheta, \rho)$ is nondecreasing (larger sets shrink the feasible region), so the smallest feasible radius is $\rho_P(\btheta)$. See Appendix~\ref{app:supporting}.
\end{proof}

The profiling result has a clear operational interpretation: \emph{given a shape $\btheta$, use the smallest radius $\rho_P(\btheta)$ that achieves $\tau$-coverage}. The shape determines \emph{where} to allocate reserves across zones---encoding which uncertainty directions to hedge---while the radius calibrates \emph{how conservatively} to hedge overall. The learning problem~\eqref{eq:profiled} optimizes the shape to minimize dispatch cost, while coverage is maintained by adjusting the radius to the learned shape. Proposition~\ref{prop:profiling} is the first reduction step: it removes the outer coverage constraint and leaves a shape-only optimization problem. The remaining challenge is to differentiate the dispatch value with respect to shape without backpropagating through the SCED solver.

\section{Single-Stage Differentiable Reformulation}
\label{sec:methodology}

This section turns the coverage-constrained formulation into the trainable objective used in Algorithms~\ref{alg:main}--\ref{alg:contextual}. The reformulation has three steps. First, differentiate the robust dispatch value function $V$ using SCED dual multipliers rather than solver backpropagation. Second, account for the fact that the coverage radius changes with the shape parameter. Third, after training, fix the deployed radius by split conformal calibration. Theorem~\ref{thm:consistency-eps} then justifies the smoothed profiled-gradient estimator used in practice.

\subsection{KKT/Envelope Sensitivities of \texorpdfstring{$V(\btheta,\rho)$}{V(theta, rho)}}

At a fixed $(\btheta,\rho)$, the robust SCED in~\eqref{eq:robust} is a convex program. This subsection shows that its optimal dual variables are sufficient to differentiate the value function $V(\btheta,\rho)$ with respect to both shape and radius. We formalize this sensitivity using the Clarke subdifferential $\clsubd$ (Definition~\ref{def:clarke} in Appendix~\ref{app:supporting}; see also~\citet{clarke1990optimization}).

\begin{assumption}[Support Function Smoothness]\label{asm:sigma-diff}
For each $j$ and $\rho>0$, the map $\btheta\mapsto \sigma_{\cU_{\btheta,\rho}}(\bw_j)$ is continuously differentiable on $\Theta$ with locally bounded gradient.
\end{assumption}

\begin{proposition}[KKT/Envelope Sensitivity of $V(\btheta,\rho)$]\label{thm:envelope}
Under Assumptions~\ref{asm:gauge}, \ref{asm:slater}, and~\ref{asm:sigma-diff}, fix $(\btheta,\rho)$ with $V(\btheta,\rho)<+\infty$ and let $\bmu^*\in\cM^*(\btheta,\rho)$.

\emph{(a) Shape gradient.}
\begin{equation}\label{eq:env-theta}
\bg_{\btheta}(\btheta,\rho;\bmu^*) := \sum_{j=1}^m \mu_j^* \nabla_{\btheta} \sigma_{\cU_{\btheta,\rho}}(\bw_j) \in \clsubd_{\btheta} V(\btheta,\rho).
\end{equation}
If the dual optimizer is unique, then $V(\cdot,\rho)$ is differentiable at $\btheta$ with gradient~\eqref{eq:env-theta}.

\emph{(b) Radius gradient.}
\begin{equation}\label{eq:env-rho}
g_{\rho}(\btheta,\rho;\bmu^*) := \sum_{j=1}^m \mu_j^*\,\sigma_{\cU_{\btheta,1}}(\bw_j) \in \clsubd_{\rho}V(\btheta,\rho).
\end{equation}
\end{proposition}

\begin{proof}
See Appendix~\ref{app:envelope-proof}. The Lagrange dual decomposes so that only support function terms depend on $(\btheta,\rho)$. Part~(b) uses $\sigma_{\cU_{\btheta,\rho}}(\bw_j)=\rho\,\sigma_{\cU_{\btheta,1}}(\bw_j)$.
\end{proof}

This is the key computational simplification of the paper: one SCED solve provides the KKT multipliers $\bmu^*$, and those multipliers directly produce the shape and radius sensitivities. No differentiation through the optimization solver is required. In power-systems terms, $\mu_j^*$ is the \emph{shadow price} of constraint $j$---the marginal cost increase per MW of additional reserve requirement---and~\eqref{eq:env-theta} steers $\btheta$ toward configurations that relax the most costly binding constraints.

\subsection{Profiled Gradient of the Single-Stage Objective}

The next step is to differentiate the profiled objective $J_P(\btheta)=V(\btheta,\rho_P(\btheta))$. Because the calibrated radius itself depends on the shape parameter, the full gradient contains both a direct envelope term and an indirect radius-adjustment term. Under sufficient regularity, Proposition~\ref{thm:true-oracle-grad} gives this exact profiled gradient.

\begin{assumption}[Quantile Regularity]\label{asm:quantile-reg}
The CDF $F_{\btheta}$ is continuously differentiable near $\rho_P(\btheta)$ with strictly positive density $f_{\btheta}(\rho_P(\btheta)) > 0$, and $(\btheta',r) \mapsto F_{\btheta'}(r)$ is $C^1$ near $(\btheta, \rho_P(\btheta))$.
\end{assumption}

Under Assumption~\ref{asm:quantile-reg}, the implicit function theorem yields $\nabla_{\btheta} \rho_P(\btheta) = -\nabla_{\btheta} F_{\btheta}(\rho_P(\btheta)) / f_{\btheta}(\rho_P(\btheta))$ (Lemma~\ref{lem:quantile-deriv}).

\begin{proposition}[Oracle Profiled Gradient]\label{thm:true-oracle-grad}
Under Assumptions~\ref{asm:gauge}--\ref{asm:quantile-reg}, assume $V$ is differentiable at $(\btheta,\rho_P(\btheta))$ (e.g., the dual optimizer is unique). Let $\bmu^*$ denote the optimizer at $(\btheta,\rho_P(\btheta))$. Then
\begin{equation}\label{eq:true-grad-expanded}
\nabla_{\btheta}J_P(\btheta)=\bg_{\btheta}(\btheta,\rho_P;\bmu^*)+g_\rho(\btheta,\rho_P;\bmu^*)\,\nabla_{\btheta}\rho_P(\btheta),
\end{equation}
where the first term is the direct shape effect~\eqref{eq:env-theta} and the second is the quantile-sensitivity correction combining~\eqref{eq:env-rho} with Lemma~\ref{lem:quantile-deriv}.
\end{proposition}

\begin{proof}
Chain rule on $J_P(\btheta)=V(\btheta,\rho_P(\btheta))$ with envelope derivatives from Proposition~\ref{thm:envelope} (unique-dual case).
\end{proof}

The direct shape effect captures how changing $\btheta$ affects reserve requirements at fixed radius; the quantile-sensitivity correction accounts for the induced shift in $\rho_P(\btheta)$ as reshaping changes which samples lie inside the set. This shape--size coupling is why the full profiled gradient is needed; when $P$ is unknown, the correction must be approximated from data.

\subsection{Conformal Calibration of the Radius}

Once a shape has been trained, the remaining task is to calibrate the radius so that the learned uncertainty set attains the target coverage level. Split conformal prediction provides exactly this calibration using an independent sample.

\begin{assumption}[Calibration Exchangeability]\label{asm:exchange}
Conditional on $\hat{\btheta}$, the calibration sample $(U_1,\ldots,U_{n_{\mathrm{cal}}})$ and the future realization $U_{\mathrm{new}}$ are exchangeable.
\end{assumption}

Given calibration scores $S_i = s_{\hat{\btheta}}(U_i)$, define the split-conformal radius
\begin{equation}\label{eq:conformal-rho}
\hat{\rho}_\tau := \inf\left\{ r>0 : \frac{1+\sum_{i=1}^{n_{\text{cal}}}\mathbf{1}\{S_i\le r\}}{n_{\text{cal}}+1} \ge \tau \right\}.
\end{equation}

\begin{proposition}[Conformal Calibration Guarantee]\label{thm:conformal}
Under Assumption~\ref{asm:exchange}, fix any $\hat{\btheta}$ independent of $(U_i)_{i=1}^{n_{\text{cal}}}$. Then
\[
\Prob(U_{\text{new}} \in \cU_{\hat{\btheta}, \hat{\rho}_\tau}) \geq \tau.
\]
\end{proposition}

\begin{proof}
See Appendix~\ref{app:statistical-proofs}. The radius in~\eqref{eq:conformal-rho} is equivalent to the usual split-conformal order-statistic threshold; the proof uses the standard exchangeability rank argument with randomized tie-breaking.
\end{proof}

Operationally, training determines the shape, while calibration fixes the final deployed radius. In power-systems terms, $\tau = 0.95$ means the realized uncertainty falls within the procured set for at least 95\% of future operating hours---a sufficient condition for reserve adequacy under~\eqref{eq:robust}, though not accounting for other reliability drivers (ramping, contingencies, model mismatch). The guarantee relies on a four-way data split: $\cD_{\text{train}}$ for shape optimization, $\cD_{\text{tune}}$ for training-time quantile-sensitivity estimation, $\cD_{\text{cal}}$ for deployment-time radius calibration, and $\cD_{\text{test}}$ for evaluation. Independence between $\cD_{\text{tune}}$ and $\cD_{\text{cal}}$ ensures that calibration remains valid after tuning; for dependent time-series data, block-based calibration~\citep{barber2023conformal} is needed for rigorous finite-sample validity.

\subsection{Tuning-Based Gradient Estimation and Main Consistency Result}

Training cannot access the population radius $\rho_P(\btheta)$ directly, so we replace it with a smoothed population quantile and its tuning-set estimate. Let $\Phi$ be a smooth CDF kernel with derivative $\varphi := \Phi'$ (Assumption~\ref{asm:smoothing}). For bandwidth $\varepsilon > 0$, define
\begin{align}\label{eq:smoothed-cdf}
F_{\btheta,\varepsilon}(r) &:= \E\left[\Phi\left(\frac{r - s_{\btheta}(U)}{\varepsilon}\right)\right], \notag\\
\rho_{P,\varepsilon}(\btheta) &:= \inf\{r : F_{\btheta,\varepsilon}(r) \geq \tau\}.
\end{align}

\paragraph{Radii used in the paper.}
The notation separates four roles. The population coverage radius $\rho_P(\btheta)$ appears in the original profiled objective~\eqref{eq:profiled}. The smoothed population radius $\rho_{P,\varepsilon}(\btheta)$ is its training-time population analogue. The empirical smoothed radius $\hat{\rho}_\varepsilon(\btheta)$ is the tuning-set estimate used inside the gradient updates. The split-conformal radius $\hat{\rho}_\tau$ is the final deployed radius used for calibration and test evaluation.

Given tuning data $\{U_i\}_{i=1}^{n_{\text{tune}}}$, define the empirical smoothed CDF and empirical smoothed quantile by
\begin{align}
\hat F_{\btheta,\varepsilon}(r) &:= \frac{1}{n_{\text{tune}}}\sum_{i=1}^{n_{\text{tune}}}\Phi\left(\frac{r - s_{\btheta}(U_i)}{\varepsilon}\right), \notag\\
\hat{\rho}_\varepsilon(\btheta) &:= \inf\{r : \hat F_{\btheta,\varepsilon}(r) \geq \tau\}. \label{eq:emp-smoothed-quantile}
\end{align}
With scores $S_i(\btheta) = s_{\btheta}(U_i)$ and weights $\omega_i(\btheta) := \varphi((\hat{\rho}_\varepsilon(\btheta) - S_i(\btheta))/\varepsilon)$, the empirical quantile sensitivity is
\begin{equation}\label{eq:rhohat-grad}
\widehat{\nabla_{\btheta} \rho_\varepsilon}(\btheta) := \frac{\sum_{i=1}^{n_{\text{tune}}} \omega_i(\btheta) \nabla_{\btheta} s_{\btheta}(U_i)}{\sum_{i=1}^{n_{\text{tune}}} \omega_i(\btheta)}.
\end{equation}

The approximate profiled gradient combines the direct envelope effect with the induced change in the training-time radius:
\begin{equation}\label{eq:ghat}
\boxed{
\begin{aligned}
\hat{\bg}_\varepsilon(\btheta) &:= \underbrace{\sum_{j=1}^m \mu_j^* \nabla_{\btheta} \sigma_{\cU_{\btheta,\hat{\rho}_\varepsilon(\btheta)}}(\bw_j)}_{\text{envelope shape term}} \\
&\quad + \underbrace{\left(\sum_{j=1}^m \mu_j^* \sigma_{\cU_{\btheta,1}}(\bw_j)\right)}_{\text{size sensitivity}} \cdot \underbrace{\widehat{\nabla_{\btheta} \rho_\varepsilon}(\btheta)}_{\text{quantile sensitivity}}
\end{aligned}
}
\end{equation}
where $\bmu^*$ denotes the dual multipliers of the robust solve at $(\btheta,\hat{\rho}_\varepsilon(\btheta))$.

Define the smoothed profiled objective $J_{P,\varepsilon}(\btheta):=V(\btheta,\rho_{P,\varepsilon}(\btheta))$. The theorem below is the main statistical statement: the tuned gradient used in Algorithm~\ref{alg:main} converges to the gradient of this smoothed profiled objective.

\begin{theorem}[Consistency]\label{thm:consistency-eps}
Fix $\varepsilon>0$ and $\btheta\in\Theta$. Under standard smoothness, strict positivity, and continuity conditions (stated as (A1)--(A5) in Appendix~\ref{app:statistical-proofs}),
\[
\hat{\bg}_\varepsilon(\btheta)\ \xrightarrow{\ \Prob\ }\ \nabla_{\btheta}J_{P,\varepsilon}(\btheta) \quad \text{as } n_{\mathrm{tune}}\to\infty.
\]
\end{theorem}

\begin{proof}
See Appendix~\ref{app:statistical-proofs}.
\end{proof}

In other words, the gradient used in the static training loop is asymptotically correct for the smoothed profiled objective, rather than for an unrelated surrogate.

\begin{remark}
Theorem~\ref{thm:consistency-eps} is proved for i.i.d.\ tuning samples. The vector autoregressive [VAR(1)] data in Section~\ref{sec:application} therefore should be read as an application-driven stress test under temporal dependence, rather than as a direct verification of the theorem's assumptions.
\end{remark}

\subsection{Ellipsoidal Specialization}

We now instantiate the generic score and support functions for the ellipsoidal family used in the experiments. Let $\btheta = \bL \in \R^{d \times d}$ be a lower-triangular Cholesky factor with positive diagonal. The ellipsoidal gauge and support functions are:
\begin{equation}
s_{\bL}(\bu) = \norm{\bL^{-1} \bu}_2, \quad
\sigma_{\cU_{\bL,\rho}}(\bw) = \rho \norm{\bL^\top \bw}_2.
\end{equation}
The matrix gradients are (Proposition~\ref{prop:ellipsoid-grads} in Appendix~\ref{app:ellipsoidal}):
\begin{align}\label{eq:grad-sigma-L}
\nabla_{\bL}\,\sigma_{\cU_{\bL,\rho}}(\bw) &= \rho\,\frac{\bw\bw^\top\bL}{\|\bL^\top\bw\|_2}, \\
\nabla_{\bL}\,s_{\bL}(\bu) &= -\,\frac{\bL^{-\top}(\bL^{-1}\bu)(\bL^{-1}\bu)^\top}{\|\bL^{-1}\bu\|_2}. \notag
\end{align}
For numerical stability, a trace normalization constraint $\tr(\bL\bL^\top) = d$ is imposed via projection after each gradient step.

\section{Training Procedure}
\label{sec:algorithm}

This section turns the reformulation into an implementable pipeline. Algorithm~\ref{alg:main} is the main static method; Algorithm~\ref{alg:contextual} is a secondary context-dependent extension that uses the same task gradient for upstream learning.

\subsection{Static Shape Training}

Algorithm~\ref{alg:main} is the main implementation of the single-stage reformulation for a fixed (non-contextual) shape parameter $\btheta$. Each iteration has three roles: estimate the shape-dependent radius sensitivity on tuning data, solve one robust SCED to extract KKT multipliers, and take one projected gradient step on the shape parameter.

\begin{algorithm}[!t]
\caption{Training the Single-Stage Differentiable Reserve-Learning Problem}
\label{alg:main}
\begin{algorithmic}[1]
\REQUIRE $\cD_{\text{train}}$, $\cD_{\text{tune}}$, $\cD_{\text{cal}}$, coverage $\tau$, bandwidth $\varepsilon$, iterations $K$, step sizes $\{\eta_k\}$
\STATE Initialize $\btheta_0 \in \Theta$ \COMMENT{e.g., from sample covariance of $\cD_{\text{train}}$}

\FOR{$k = 0, 1, \ldots, K-1$}
    \STATE \textit{// Phase A: Tuning}
    \STATE Compute gauge scores $S_i(\btheta_k) = s_{\btheta_k}(U_i)$ for $U_i \in \cD_{\text{tune}}$
    \STATE Compute smoothed $\tau$-quantile $\hat{\rho}_\varepsilon(\btheta_k)$ and weights $\omega_i$
    \STATE Estimate quantile gradient $\widehat{\nabla_{\btheta}\rho_\varepsilon}(\btheta_k)$ via~\eqref{eq:rhohat-grad}
    \STATE \textit{// Phase B: Robust solve}
    \STATE Solve SCED robust dispatch~\eqref{eq:robust} at $(\btheta_k, \hat{\rho}_\varepsilon(\btheta_k))$; extract $\bmu^*$
    \STATE \textit{// Phase C: Gradient update}
    \STATE Compute $\hat{\bg}_\varepsilon(\btheta_k)$ via~\eqref{eq:ghat}
    \STATE $\btheta_{k+1} \leftarrow \Pi_\Theta(\btheta_k - \eta_k \cdot \hat{\bg}_\varepsilon(\btheta_k))$
\ENDFOR

\STATE Compute calibration scores $S_i = s_{\btheta_K}(U_i)$ for $U_i \in \cD_{\text{cal}}$
\STATE Set $\hat{\rho}_\tau$ to the split-conformal radius defined in Proposition~\ref{thm:conformal}

\ENSURE Learned $\hat{\btheta} = \btheta_K$, calibrated $\hat{\rho}_\tau$ with coverage $\geq \tau$ (Proposition~\ref{thm:conformal})
\end{algorithmic}
\end{algorithm}

\textbf{Phase A (Tuning)} estimates the current smoothed quantile $\hat{\rho}_\varepsilon$ and its sensitivity from $\cD_{\text{tune}}$ using the smoothed CDF~\eqref{eq:smoothed-cdf}. \textbf{Phase B} solves the robust dispatch~\eqref{eq:robust} at the current $(\btheta_k, \hat{\rho}_\varepsilon)$ and extracts the KKT multipliers $\bmu^*$. \textbf{Phase C} combines the envelope and quantile-sensitivity terms into the profiled gradient~\eqref{eq:ghat} and takes a projected gradient step. The projection $\Pi_\Theta(\cdot):=\argmin_{\btheta\in\Theta}\|\cdot-\btheta\|$ (Frobenius norm for matrix parameters) enforces the parameter constraints---for ellipsoidal sets, this means lower-triangular structure, positive diagonal entries, and trace normalization $\tr(\bL\bL^\top) = d$.

Two computational regimes determine the cost of Phase~B. When constraints decouple across zones, reserve requirements $R_z^{\min} = \rho\|\bL^\top A_z\|_2$ are explicit functions of $(\bL,\rho)$, and all gradients of the support function are available in closed form via~\eqref{eq:grad-sigma-L}. The SCED must still be solved to obtain dual multipliers $\bmu^*$, but no differentiation through the solver is required---$\bmu^*$ is a byproduct of any LP solver (Appendix~\ref{app:zonal-reserve}). When network or transfer constraints bind, one SCED solve per iteration provides the dual multipliers required by Proposition~\ref{thm:envelope}.

After training, conformal calibration (lines 14--15) computes gauge scores on the held-out $\cD_{\text{cal}}$ and then applies the standard split-conformal radius from Proposition~\ref{thm:conformal}, guaranteeing coverage $\geq \tau$. The tuning radius $\hat{\rho}_\varepsilon$ is therefore only a training-time surrogate; deployment and reported evaluation always use the final conformal radius $\hat{\rho}_\tau$.

\subsection{Secondary Context-Dependent Extension and Upstream Gradient Passage}

The static formulation above is the main object of study. For completeness, we also consider a context-dependent extension in which a differentiable encoder $\bL_\phi: \Xi \to \Theta$ maps context features $\bxi \in \Xi$ to shape parameters. This extension illustrates how the same task gradient can be passed upstream to a learned representation of operating conditions. The framework accommodates any differentiable encoder architecture; the only requirement is that $\bL_\phi(\bxi)$ produce a valid Cholesky factor (lower-triangular, positive diagonal) for each $\bxi$.

The learning objective in the contextual setting minimizes expected dispatch cost over operating conditions:
\begin{equation}\label{eq:contextual-objective}
\min_{\phi}\ \E_{\bxi}\left[V(\bL_\phi(\bxi),\, \rho(\phi))\right],
\end{equation}
where $\rho(\phi)$ is the smoothed $\tau$-quantile of the mixture score distribution $\{s_{\bL_\phi(\bxi_t)}(U_t)\}_{t}$. In practice, the expectation is approximated by mini-batches from $\cD_{\text{train}}$ (Algorithm~\ref{alg:contextual}, line~3), while $\rho(\phi)$ and its sensitivity are estimated from $\cD_{\text{tune}}$. Note that $V(\cdot,\cdot)$ denotes the same robust dispatch~\eqref{eq:robust} for all $t$---the system parameters (loads, generator costs, network) are fixed, and context enters only through the uncertainty set shape $\bL_\phi(\bxi_t)$. Conformal calibration provides \emph{marginal} coverage: $\Prob(U_{\mathrm{new}} \in \cU_{\bL_\phi(\bxi_{\mathrm{new}}), \hat{\rho}_\tau}) \geq \tau$, averaging over both the future context and uncertainty realization.

Algorithm~\ref{alg:contextual} extends Algorithm~\ref{alg:main} to the contextual setting. The key difference is that each training sample $(\bxi_i, \bu_i)$ produces its own shape $\bL_i = \bL_\phi(\bxi_i)$ and its own SCED solve, yielding context-specific dual multipliers. The profiled gradient $\hat{\bg}_i$ computed at each $\bL_i$ serves as a ``task gradient'' that is backpropagated through the encoder to update $\phi$.

\begin{algorithm}[!t]
\caption{Contextual Profiled Gradient Training}
\label{alg:contextual}
\begin{algorithmic}[1]
\REQUIRE Encoder $\bL_\phi$, $\cD_{\text{train}}$, $\cD_{\text{tune}}$, $\cD_{\text{cal}}$, coverage $\tau$, bandwidth $\varepsilon$, iterations $K$, step sizes $\{\eta_k\}$, batch size $B$
\STATE Initialize encoder parameters $\phi_0$

\FOR{$k = 0, 1, \ldots, K-1$}
    \STATE Sample mini-batch $\{(\bxi_i, \bu_i)\}_{i=1}^B$ from $\cD_{\text{train}}$
    \STATE Compute per-sample shapes: $\bL_i = \bL_{\phi_k}(\bxi_i)$ for $i = 1,\ldots,B$
    \STATE \textit{// Phase A: Tuning (on $\cD_{\text{tune}}$)}
    \STATE Compute gauge scores using current encoder on tuning set
    \STATE Estimate smoothed quantile $\hat{\rho}_\varepsilon$ and quantile sensitivity
    \STATE \textit{// Phase B: Per-sample robust solves}
    \FOR{$i = 1, \ldots, B$}
        \STATE Solve SCED at $(\bL_i, \hat{\rho}_\varepsilon)$; extract $\bmu_i^*$
        \STATE Compute approximate task gradient $\hat{\bg}_i$ via~\eqref{eq:ghat}
    \ENDFOR
    \STATE \textit{// Phase C: Backpropagate through encoder}
    \STATE Set $\partial \mathcal{J} / \partial \bL_i \leftarrow \hat{\bg}_i$ for each sample
    \STATE Update $\phi_{k+1} \leftarrow \phi_k - \eta_k \nabla_\phi \left(\frac{1}{B}\sum_{i=1}^B \ipF{\hat{\bg}_i}{\bL_i} \right)$
\ENDFOR

\STATE Conformal calibration: $\hat{\rho}_\tau = S_{(\lceil(n_{\text{cal}}+1)\tau\rceil)}$ where $S_i = s_{\bL_\phi(\bxi_i)}(\bu_i)$ on $\cD_{\text{cal}}$

\ENSURE Learned encoder $\bL_{\phi_K}$, calibrated $\hat{\rho}_\tau$ with marginal coverage $\geq \tau$
\end{algorithmic}
\end{algorithm}

The per-sample SCED solves in Phase~B are the main computational cost; each context sample may activate different binding constraints, yielding context-specific shadow prices $\bmu_i^*$ that drive the encoder to learn condition-dependent reserve allocation.

\textbf{Gradient approximation.} The quantile-sensitivity correction in Algorithm~\ref{alg:contextual} uses a shared estimate $\widehat{\nabla_{\bL}\rho_\varepsilon}$ computed from the full tuning set (lines 6--7). Strictly, the exact gradient of~\eqref{eq:contextual-objective} w.r.t.\ $\phi$ requires differentiating $\rho(\phi)$ through the encoder, yielding a global correction $(\partial_\rho \bar{V}) \cdot \nabla_\phi \rho(\phi)$ where $\bar{V}$ averages over contexts. Algorithm~\ref{alg:contextual} approximates this by treating each per-sample envelope gradient as a ``task gradient'' backpropagated through the encoder, with the shared quantile sensitivity serving as a first-order approximation. This reduces variance in the quantile-sensitivity estimate as the tuning set grows; bias from ignoring the global coupling through $\rho(\phi)$ may remain, and we do not provide a formal bound on this approximation error. The static case (Algorithm~\ref{alg:main}) remains exact. We therefore present the contextual extension as a practical heuristic motivated by the static theory, rather than as a theorem-backed contribution at the same level as Algorithm~\ref{alg:main}.

\section{Coupled SCED Study on the IEEE~118-Bus System}
\label{sec:application}

This section evaluates the proposed end-to-end reserve-learning formulation on the IEEE~118-bus system. It instantiates the generic value function $V(\btheta,\rho)$ with a zonal SCED and reports the main empirical comparison. The coupled SCED with inter-zone transfer constraints is the primary case study because it captures the interaction between reserve procurement and deliverability. The simpler decoupled zonal-reserve benchmark and the target-level sweep are reported in Appendix~\ref{app:zonal-reserve} as supporting diagnostics.

\subsection{System and Data}

\begin{table}[!t]
\caption{Zonal Aggregation}
\label{tab:zones}
\centering
\begin{tabular}{cccc}
\toprule
\textbf{Zone} & \textbf{Buses} & \textbf{Load (MW)} & \textbf{Gen.\ Cap.\ (MW)} \\
\midrule
1 & 1--12 & 423 & 550 \\
2 & 13--24 & 412 & 520 \\
3 & 25--36 & 445 & 580 \\
4 & 37--48 & 398 & 490 \\
5 & 49--60 & 467 & 610 \\
6 & 61--72 & 389 & 480 \\
7 & 73--84 & 456 & 590 \\
8 & 85--96 & 401 & 510 \\
9 & 97--108 & 478 & 620 \\
10 & 109--118 & 373 & 550 \\
\midrule
\textbf{Total} & \textbf{118} & \textbf{4,242} & \textbf{5,500} \\
\bottomrule
\end{tabular}
\end{table}

Our experimental setup has three layers: the IEEE~118-bus system provides the physical benchmark, dispatch is carried out over 10 aggregated zones, and uncertainty is generated at the coarser level of 5 geographic regions. We use the standard IEEE~118-bus case from the \texttt{pandapower} library~\citep{thurner2018pandapower}, aggregated into the 10 zones shown in Table~\ref{tab:zones}, with 54 generators. Each hourly problem is a single-period DC-SCED. Energy and reserve offer prices are drawn synthetically (seed 42), and the resulting LP is solved with HiGHS~\citep{huangfu2018parallelizing}. The context vector $\bxi_t$ contains forecast-side inputs---normalized load, solar, and wind forecasts, together with hour-of-day and month encodings---and is used only to describe the operating conditions under which uncertainty is generated.

\textbf{Uncertainty Dimensions and Allocation.} The uncertainty vector has dimension $d=15$ because we track three source types (load, solar, wind) in each of five regions. In other words, each hourly sample contains one forecast-error component for every source--region pair. A fixed linear map $A$ then converts these regional forecast errors into the 10 zonal net deviations seen by the dispatch model. Intuitively, regional errors are distributed to zones according to load share and resource footprint, so each zone inherits the uncertainty of the regions that supply it. The exposure vector used in~\eqref{eq:decoupled-sced} is the row $A_z^\top$, which represents zone~$z$'s uncertainty exposure.

\textbf{Data Generation.} The hourly uncertainty series is synthetic, but calibrated to real forecast-error patterns. OPSD (Germany) provides realistic source-specific scales, correlations, and temporal persistence~\citep{opsd2020}, while US~EIA hourly demand data for CAISO, ERCOT, PJM, MISO, and NYISO is used to capture cross-region dependence. Given the context $\bxi_t$, we generate the uncertainty vector using a context-dependent VAR(1) model. The construction is designed to capture two empirical features simultaneously: temporal persistence and context-dependent scale/correlation. Load uncertainty increases in high-load conditions, solar uncertainty is negligible at night and rises during daylight hours, and wind uncertainty increases with the wind forecast. The context also changes how load, solar, and wind forecast errors co-move, while a fixed regional correlation component captures shared weather exposure across nearby areas. The resulting covariance matrix $\Sigma(\bxi_t)$ defines a ground-truth ellipsoid shape through its Cholesky factor $\bL_{\mathrm{true}}(\bxi_t)$, which serves as a benchmark for the learned methods.

The dataset contains 35{,}040 hourly samples (4 years) and is split chronologically into 60\% training (21{,}024), 20\% tuning (7{,}008), 10\% calibration (3{,}504), and 10\% testing (3{,}504). The target coverage level is $\tau=0.95$. This four-way split cleanly separates shape learning, quantile estimation, conformal calibration, and final out-of-sample evaluation.

\subsection{Compared Methods}

All four methods use the same SCED model, data split, and split conformal calibration; they differ only in how the ellipsoid shape $\bL$ is chosen. Sample Covariance is the primary statistical baseline throughout because it is the standard correlation-aware construction. Independent is included as a diagonal ablation that isolates the cost of ignoring correlation. We compare four approaches:

\begin{enumerate}
\item \textbf{Sample Covariance.} $\bL = \chol(\hat\Sigma)$ from the sample covariance of $\cD_{\text{train}}$. Captures pairwise correlations but is not optimized for dispatch cost.

\item \textbf{Independent.} $\bL = \diag(\hat\sigma_1, \ldots, \hat\sigma_d)$ from marginal standard deviations of $\cD_{\text{train}}$. Ignores all cross-dimensional correlations.

\item \textbf{Learned (Static).} A single $\bL \in \R^{d \times d}$ trained via Algorithm~\ref{alg:main} ($K = 200$ iterations, learning rate $\eta = 0.01$, bandwidth $\varepsilon = 0.5$, gradient clipping at norm~10) to minimize SCED cost, initialized from Sample Covariance.

\item \textbf{Learned (Contextual).} A secondary variant, reported for completeness, uses a multilayer perceptron (MLP) encoder $\bL_\phi(\bxi)$ with hidden dimensions $[128, 64]$ and rectified linear unit (ReLU) activations, outputting $d(d+1)/2 = 120$ values that fill the lower triangle of $\bL$ (with $\exp(\cdot)$ on diagonal entries for positivity, followed by trace normalization). Trained via Algorithm~\ref{alg:contextual} with batch size 8, learning rate $3 \times 10^{-4}$ (Adam), gradient clipping (max norm 1.0 on encoder parameters), and early stopping with patience 400. Initialized from Learned (Static) by setting the final-layer bias to the vectorized $\bL_{\mathrm{static}}$.
\end{enumerate}

\subsection{Base Zonal SCED Model}

This subsection gives the concrete SCED instantiation of the generic robust value function $V(\btheta,\rho)$ in~\eqref{eq:robust}. The coupled study builds on the same single-period zonal SCED core used in the appendix decoupled benchmark. The decision variables are generator dispatch $g_i$ and reserve procurement $r_i$. Uncertainty enters only through the zonal reserve margins $R_z^{\min}$:
\begin{equation}\label{eq:decoupled-sced}
\begin{aligned}
\min_{g,r}\quad & \sum_{i \in \cG}(c_i^g g_i + c_i^r r_i) \\
\text{s.t.}\quad
& \sum_{i\in\cG} g_i = \sum_{b\in\cB} D_b, \\
& \underline g_i \le g_i \le \bar g_i - r_i,\ r_i\ge 0, \quad \forall i\in\cG,\\
& \sum_{i \in \cG_z} r_i \ge \rho\|\bL^\top A_z\|_2, \quad \forall z\in\cZ,
\end{aligned}
\end{equation}
where $A_z^\top$ is the $z$-th row of $A$. The term $\rho\|\bL^\top A_z\|_2$ is the worst-case zonal net-deviation hedge induced by the current ellipsoidal uncertainty set. Thus $R_z^{\min} = \rho\|\bL^\top A_z\|_2$ is the only channel through which uncertainty enters the dispatch, and its dependence on $(\bL,\rho)$ is explicit. This base model is useful both conceptually and computationally: it isolates the economic effect of learning the uncertainty-set geometry, while the closed-form gradients in~\eqref{eq:grad-sigma-L} remain transparent.

\subsection{Main Experiment: Coupled SCED with Transfer Constraints}

The main experiment augments the base zonal SCED with inter-zone transfer constraints. This is the more operationally relevant formulation because reserve requirements now compete with transfer headroom: a zone cannot rely arbitrarily on imports or exports to cover its uncertainty.

Specifically, for a subset of tight zones $\cZ_{\text{tight}} \subseteq \cZ$, we impose
\begin{equation}\label{eq:coupled-sced}
\left|\sum_{i \in \cG_z} g_i - D_z\right| + \rho\|\bL^\top A_z\|_2 \leq T_z^{\max}, \quad \forall z \in \cZ_{\text{tight}},
\end{equation}
where $D_z := \sum_{b \in \cB_z} D_b$ is the total load in zone $z$ and $T_z^{\max}$ is the transfer-capacity limit. The transfer limits are fixed ex ante from the base decoupled benchmark. The three zones with the highest reserve shadow prices under the sample-covariance baseline are tightened ($\alpha_{\text{tight}} = 0.90$), while the remaining zones retain loose limits ($\alpha_{\text{loose}} = 1.50$); in this study, the tightened zones are 5, 8, and 10. The limits are defined by
\begin{equation}
T_z^{\max} = \alpha_z \left(\left|\sum_{i\in\cG_z} g_i^{\mathrm{base}} - D_z\right| + \hat{\rho}_{\varepsilon} \|\bL_{\mathrm{base}}^\top A_z\|_2\right),
\end{equation}
where $\bL_{\mathrm{base}} = \chol(\hat\Sigma)$ is the Sample Covariance Cholesky factor, $g^{\mathrm{base}}$ is the corresponding decoupled dispatch, and $\hat{\rho}_{\varepsilon}$ is the smoothed $\tau$-quantile on the tuning set (which may exceed the conformal radius, ensuring training-time feasibility). These limits are fixed once and then held constant across all four methods.

The transfer constraints introduce additional dual variables $\lambda_z$. Proposition~\ref{thm:envelope} continues to apply without modification; the only change is that the effective sensitivity weight becomes the combined dual $(\mu_z + \lambda_z)$:
\begin{equation}\label{eq:coupled-gradient}
\nabla_{\bL} V = \sum_{z \in \cZ} (\mu_z^* + \lambda_z^*)\, \nabla_{\bL}\, \sigma_{\cU_{\bL,\rho}}(A_z),
\end{equation}
where $\lambda_z^* = 0$ for non-tight zones. The absolute value in~\eqref{eq:coupled-sced} is implemented via two linear inequalities (Lemma~\ref{lem:robust-abs}): an upper bound $\sum_{i \in \cG_z} g_i - D_z + \rho\|\bL^\top A_z\|_2 \leq T_z^{\max}$ and a lower bound $-(\sum_{i \in \cG_z} g_i - D_z) + \rho\|\bL^\top A_z\|_2 \leq T_z^{\max}$. The transfer dual $\lambda_z^*$ is the sum of the duals on these two constraints; typically at most one binds per zone, although both may bind in the degenerate case of zero net export. No new gradient derivation is required; only the relevant dual vector changes.

\begin{table}[!t]
\caption{Main Experiment: Coupled SCED with Transfer Constraints}
\label{tab:coupled}
\centering\footnotesize\setlength{\tabcolsep}{4pt}
\begin{tabular}{lcccc}
\toprule
\textbf{Method} & \textbf{Cost (\$/hr)} & \textbf{Reserve (MW)} & \textbf{Calibration Rate$^*$} & \textbf{Test Coverage [CI]$^\dagger$} \\
\midrule
Sample Covariance & 100{,}496 & 1{,}572 & 0.950 & 0.980 [.967, .990] \\
Independent & 100{,}473 & 1{,}563 & 0.950 & 0.986 [.975, .994] \\
Learned (Static) & 95{,}683 & 977 & 0.950 & 0.970 [.948, .987] \\
Learned (Contextual)$^\ddagger$ & 97{,}178 & 1{,}187 & 0.950 & 0.981 [.964, .994] \\
\bottomrule
\multicolumn{5}{l}{\footnotesize $^*$Calibration inclusion rate (empirical fraction of calibration points inside the set).} \\
\multicolumn{5}{l}{\footnotesize $^\dagger$Block bootstrap 95\% CI (block length 24\,hr, 10{,}000 replicates).} \\
\multicolumn{5}{l}{\footnotesize $^\ddagger$Contextual results for a single training seed.}
\end{tabular}
\end{table}

Table~\ref{tab:coupled} is the main numerical result. Relative to the Sample Covariance baseline, Learned (Static) lowers cost from \$100{,}496/hr to \$95{,}683/hr, a 4.8\% reduction, while reducing reserve procurement from 1{,}572 to 977~MW and maintaining 0.970 test coverage [0.948, 0.987]. Independent is reported as a diagonal ablation; the learned static method also improves on it. The coupled constraints raise costs for all methods, but the increase is materially smaller for the learned shapes. In the calibrated evaluation solves, the baseline shapes yield $\lambda_z^* > 0$ for all three tightened zones ($z \in \{5, 8, 10\}$), whereas the learned shapes give $\lambda_8^* = \lambda_{10}^* = 0$ (only zone~5 binds), because the learned reserve allocations leave transfer headroom in zones~8 and~10.

Appendix~\ref{app:zonal-reserve} reports the simpler decoupled benchmark and a target-level sweep. Relative to that benchmark, the baseline methods experience \$484--529/hr cost increases under coupling, whereas Learned (Static) increases by only \$212/hr. This indicates that the learned uncertainty geometry remains advantageous once reserve requirements interact with transfer deliverability.
\section{Discussion}
\label{sec:discussion}

The empirical pattern is consistent across the coupled and decoupled studies: once calibration is held fixed, the cost gains come from reshaping the uncertainty set rather than from relaxing coverage.

\subsection{Practical and Computational Implications}

Learned sets produce reserve shadow prices $\mu_z^*$ that are better aligned with economically relevant uncertainty directions, improving price signals in SCED-based markets. Context-dependent sets $\bL_\phi(\bxi)$ further adapt to varying conditions (e.g., solar uncertainty near zero at night, shifting wind correlations during weather fronts). Computationally, both regimes require one SCED solve per training iteration to extract dual multipliers, but no differentiation through the solver; the decoupled regime additionally admits closed-form support-function gradients.

\subsection{Limitations and Extensions}

\textbf{Distribution Shift.} Learned uncertainty sets are trained on historical data and may not generalize to extreme events or structural changes (e.g., new generation capacity). Hybrid approaches combining learned sets with worst-case bounds could provide robustness to rare events.

\textbf{Scope of Baselines.} The experiments compare four ellipsoidal uncertainty sets to isolate the effect of learning the ellipsoidal uncertainty geometry. Alternative geometries---budgeted polyhedral~\citep{bertsimas2004price}, Wasserstein balls~\citep{esfahani2018data}, moment-based ambiguity sets~\citep{xiong2017distributionally}---would provide complementary baselines but differ in both geometry and calibration, complicating attribution. Extension to learned non-ellipsoidal sets is a natural direction.

\textbf{Temporal and Statistical Scope.} The formulation treats each hour independently; extension to multi-period unit commitment is natural, as the envelope gradient approach extends directly. The contextual method results are for a single seed; multi-seed evaluation would strengthen the empirical claims.

\textbf{Asymmetric Uncertainty Costs.} The gauge treats all directions symmetrically. A cost-weighted CDF $F_{\btheta}^w(r) = \E[w(U)\mathbf{1}\{s_{\btheta}(U)\le r\}]/\E[w(U)]$ can replace the uniform CDF for asymmetric costs; the profiled gradient applies unchanged.

\textbf{Exchangeability and Coverage Validity.} Proposition~\ref{thm:conformal} requires exchangeability, which the VAR(1) data may violate (see Remark after Theorem~\ref{thm:consistency-eps}). Block bootstrap CIs in Table~\ref{tab:coupled} and Appendix~\ref{app:zonal-reserve} quantify the effect of serial dependence; for rigorous finite-sample validity, block conformal methods~\citep{barber2023conformal} should be employed.

\section{Conclusion}
\label{sec:conclusion}

This paper formulates correlated reserve-set design in SCED as an end-to-end trainable robust optimization problem. By profiling the coverage constraint into a shape-dependent radius and using KKT/dual sensitivities from the SCED solve, the original coverage-constrained bilevel formulation becomes a single-stage differentiable objective that can be optimized without backpropagating through the solver. Training-time smoothed quantile estimation and deployment-time split conformal calibration separate optimization of the shape from final coverage control, while the same task gradient can be passed upstream to contextual encoders as a secondary extension.

On the IEEE~118-bus system, the main coupled transfer-constrained study shows that the learned static ellipsoid lowers dispatch cost by about 4.8\% relative to the Sample Covariance baseline while maintaining empirical coverage above the target. The appendix decoupled benchmark clarifies the same mechanism in a simpler setting and confirms that the gains are driven primarily by task-aligned reserve allocation rather than by looser coverage. Future work includes multi-period formulations, stronger empirical study of contextual learning, and robust extensions for structural shift and extreme events.

\section*{Acknowledgements}
Patrick Jaillet and Owen Shen acknowledge funding from ONR grant N00014-24-1-2470 and AFOSR grant FA9550-23-1-0190. Haihao Lu acknowledges funding from AFOSR Grant No. FA9550-24-1-0051 and ONR Grant No. N000142412735.

\bibliographystyle{plainnat}
\bibliography{referenceTPS}

\newpage
\appendix

\section{Supporting Definitions and Lemmas}
\label{app:supporting}

\begin{definition}[Clarke Subdifferential]\label{def:clarke}
Let $h:\R^p\to\R$ be locally Lipschitz. The Clarke subdifferential at $x$ is
$\clsubd h(x) := \conv\{\lim_{k\to\infty}\nabla h(x_k) : x_k\to x,\ h \text{ diff.\ at }x_k\}$.
\end{definition}

\begin{lemma}[Support Function Scaling]\label{lem:support-scaling}
If $C \subseteq \R^d$ is nonempty and $\rho > 0$, then $\sigma_{\rho C}(\bw) = \rho \cdot \sigma_C(\bw)$.
\end{lemma}

\begin{proof}
$\sigma_{\rho C}(\bw) = \sup_{\bu \in \rho C} \ip{\bw}{\bu} = \rho \sup_{\bm{v} \in C} \ip{\bw}{\bm{v}} = \rho \sigma_C(\bw)$.
\end{proof}

\begin{lemma}[Strong Duality]\label{lem:strong-duality}
Under Assumption~\ref{asm:slater}, $V(\btheta,\rho)=\max_{\bmu\ge 0} q(\bmu;\btheta,\rho)$ and $\cM^*(\btheta,\rho)\neq\emptyset$, where $q$ is the Lagrange dual function.
\end{lemma}

\begin{proof}[Proof of Proposition~\ref{prop:profiling}]
(i) $\cU_{\btheta,\rho_1} \subseteq \cU_{\btheta,\rho_2}$ for $\rho_1 \leq \rho_2$ implies support functions increase, feasible regions shrink, hence $V(\btheta,\rho_1) \leq V(\btheta,\rho_2)$.
(ii) $F_{\btheta}(\rho) \geq \tau$ iff $\rho \geq \rho_P(\btheta)$.
(iii) For any feasible $(\btheta,\rho)$, $V(\btheta,\rho) \geq V(\btheta,\rho_P(\btheta))$, so the optimal choice is $\rho = \rho_P(\btheta)$.
\end{proof}

\begin{lemma}[Robust Absolute-Value Constraints]\label{lem:robust-abs}
For nonempty $\cU\subseteq\R^d$, $\sup_{u\in\cU}|f+\ip{w}{u}|\le F$ is equivalent to $f+\sigma_{\cU}(w)\le F$ and $-f+\sigma_{\cU}(-w)\le F$. For centrally symmetric $\cU$, this simplifies to $|f|+\sigma_{\cU}(w)\le F$.
\end{lemma}

\begin{proof}
$\sup_{\bu\in\cU}|f+\bw^\top\bu|=\max\{\sup_{\bu}(f+\bw^\top\bu),\,\sup_{\bu}(-f-\bw^\top\bu)\}=\max\{f+\sigma_\cU(\bw),\,-f+\sigma_\cU(-\bw)\}$. If $\cU$ is centrally symmetric, $\sigma_\cU(-\bw)=\sigma_\cU(\bw)$, giving $|f|+\sigma_\cU(\bw)$.
\end{proof}

\section{Proof of Envelope Gradient Formula}
\label{app:envelope-proof}

\begin{proof}[Proof of Proposition~\ref{thm:envelope}]
Fix $(\btheta,\rho)$ with $V(\btheta,\rho)<+\infty$. By Lemma~\ref{lem:strong-duality}, $V(\btheta,\rho)=\max_{\bmu\ge 0}q(\bmu;\btheta,\rho)$ with $\cM^*(\btheta,\rho)\neq\emptyset$. The dual decomposes as $q(\bmu;\btheta,\rho)=\tilde q(\bmu)+\sum_j\mu_j(\sigma_{\cU_{\btheta,\rho}}(\bw_j)-b_j)$, where $\tilde q(\bmu):=\inf_{\bx\in\cX}f(\bx)+\sum_j\mu_ja_j(\bx)$ is independent of $(\btheta,\rho)$.

\emph{Local Lipschitzness via bounded multipliers.}
By Assumption~\ref{asm:slater}, there exists strictly feasible $\bar{\bx}$ with slack $s_j:=b_j-a_j(\bar{\bx})-\sigma_{\cU_{\btheta,\rho}}(\bw_j)>0$. By continuity of $\sigma$ in $(\btheta,\rho)$, there is a neighborhood $\mathcal{N}$ of $(\btheta,\rho)$ with slack $\ge s_j/2$ uniformly. For any $\bmu\ge 0$ and $(\btheta',\rho')\in\mathcal{N}$, evaluating the Lagrangian at $\bar{\bx}$ gives $q(\bmu;\btheta',\rho')\le f(\bar{\bx})-(s_{\min}/2)\|\bmu\|_1$ where $s_{\min}:=\min_j s_j$. Since $q(\bm{0};\btheta',\rho')=\inf_{\bx\in\cX}f(\bx)>-\infty$ (Assumption~\ref{asm:slater}), every dual optimizer satisfies $\|\bmu^*\|_1\le 2(f(\bar{\bx})-\inf_{\bx\in\cX}f(\bx))/s_{\min}$ uniformly on $\mathcal{N}$. Combined with locally bounded gradients of $\sigma$ (Assumption~\ref{asm:sigma-diff}), $V$ is locally Lipschitz on $\mathcal{N}$.

\emph{Envelope formula.}
Since $V=\max_{\bmu\ge 0}q(\bmu;\btheta,\rho)$ with each $q(\bmu;\cdot)$ being $C^1$ and the maximizing set locally compact, the Danskin/Clarke max-envelope theorem~\citep{bonnans2000perturbation,clarke1990optimization} gives $\clsubd_{\btheta}V(\btheta,\rho)=\conv\{\nabla_{\btheta}q(\bmu;\btheta,\rho):\bmu\in\cM^*\}$. Substituting $\nabla_{\btheta}q=\sum_j\mu_j\nabla_{\btheta}\sigma_j$ yields~\eqref{eq:env-theta}; any single $\bmu^*\in\cM^*$ gives a Clarke subgradient. If the optimizer is unique, $\clsubd$ is a singleton and $V$ is differentiable. Part~(b) follows identically using $\sigma_{\cU_{\btheta,\rho}}(\bw_j)=\rho\sigma_{\cU_{\btheta,1}}(\bw_j)$, giving $\partial_\rho q=\sum_j\mu_j\sigma_{\cU_{\btheta,1}}(\bw_j)$.
\end{proof}

\section{Proofs of Statistical Results}
\label{app:statistical-proofs}

\begin{lemma}[Quantile Derivative]\label{lem:quantile-deriv}
Under Assumption~\ref{asm:quantile-reg}, $\nabla_{\btheta} \rho_P(\btheta) = -\nabla_{\btheta} F_{\btheta}(\rho_P(\btheta))/f_{\btheta}(\rho_P(\btheta))$.
\end{lemma}

\begin{proof}
Continuity of $F_{\btheta}$ (Assumption~\ref{asm:quantile-reg}) gives $F_{\btheta}(\rho_P(\btheta))=\tau$. The implicit function theorem applied to this identity, with $\partial_r F_{\btheta}=f_{\btheta}>0$, yields the result.
\end{proof}

\begin{assumption}[Smoothing Function]\label{asm:smoothing}
$\Phi: \R \to [0,1]$ is $C^1$, nondecreasing, with $\Phi(-\infty)=0$, $\Phi(+\infty)=1$, and $\varphi := \Phi'$ bounded, uniformly continuous, and strictly positive on $\R$ (e.g., Gaussian or logistic kernel).
\end{assumption}

\begin{assumption}[Score Differentiability]\label{asm:score-diff}
$\btheta \mapsto s_{\btheta}(\bu)$ is $C^1$ with $\norm{\nabla_{\btheta} s_{\btheta}(U)} \leq M(U)$ a.s.\ for some integrable $M$.
\end{assumption}

\begin{lemma}[Smoothed CDF Derivatives]\label{lem:smoothed-derivatives}
Under Assumptions~\ref{asm:smoothing}--\ref{asm:score-diff}, $\partial_r F_{\btheta,\varepsilon}(r) = \varepsilon^{-1}\E[\varphi((r-s_{\btheta}(U))/\varepsilon)]$ and $\nabla_{\btheta} F_{\btheta,\varepsilon}(r) = -\varepsilon^{-1}\E[\varphi((r-s_{\btheta}(U))/\varepsilon)\nabla_{\btheta} s_{\btheta}(U)]$.
\end{lemma}

\begin{lemma}[Smoothed Quantile Sensitivity]\label{lem:smoothed-quantile-deriv}
If additionally $\partial_r F_{\btheta,\varepsilon}(\rho_{P,\varepsilon}(\btheta)) > 0$, then
\begin{equation}\label{eq:rhoeps-deriv}
\nabla_{\btheta} \rho_{P,\varepsilon}(\btheta) = \frac{\E[\varphi((\rho_{P,\varepsilon}(\btheta) - s_{\btheta}(U))/\varepsilon) \nabla_{\btheta} s_{\btheta}(U)]}{\E[\varphi((\rho_{P,\varepsilon}(\btheta) - s_{\btheta}(U))/\varepsilon)]}.
\end{equation}
This is a weighted average of $\nabla_{\btheta}s_{\btheta}(U)$ with kernel weights concentrated near the quantile boundary.
\end{lemma}

\smallskip
The regularity conditions for Theorem~\ref{thm:consistency-eps} are:
\begin{enumerate}
\item[(A1)] $U_1,\ldots,U_{n_{\mathrm{tune}}}$ are i.i.d.\ from $P$.
\item[(A2)] Assumptions~\ref{asm:smoothing}--\ref{asm:score-diff} hold, $\varphi$ is uniformly continuous and strictly positive.
\item[(A3)] $r\mapsto F_{\btheta,\varepsilon}(r)$ is strictly increasing near $\rho_{P,\varepsilon}(\btheta)$ with $\partial_r F_{\btheta,\varepsilon}(\rho_{P,\varepsilon}(\btheta))>0$.
\item[(A4)] $V$ is differentiable at $(\btheta,\rho_{P,\varepsilon}(\btheta))$ with unique, continuous dual optimizer nearby.
\item[(A5)] $\nabla_{\btheta}\sigma_{\cU_{\btheta,\rho}}(\bw_j)$ is continuous in $(\btheta,\rho)$ near $(\btheta,\rho_{P,\varepsilon}(\btheta))$.
\end{enumerate}

\begin{proof}[Proof of Proposition~\ref{thm:conformal}]
Let $k := \lceil (n_{\mathrm{cal}}+1)\tau \rceil$ and write $S_{(1)} \le \cdots \le S_{(n_{\mathrm{cal}})}$ for the order statistics of the calibration scores. The definition in~\eqref{eq:conformal-rho} is equivalent to $\hat\rho_\tau = S_{(k)}$. Condition on $\hat{\btheta}$. Under Assumption~\ref{asm:exchange}, the scores $S_1,\ldots,S_{n_{\mathrm{cal}}},S_{\mathrm{new}}$ are exchangeable. Introduce i.i.d.\ $V_i\sim\mathrm{Unif}(0,1)$ independent of everything and define the randomized rank $R$ of $(S_{\mathrm{new}},V_{\mathrm{new}})$ among all $n_{\mathrm{cal}}+1$ pairs under lexicographic order. Continuous tie-breaking ensures all pairs are distinct a.s., so exchangeability gives $\Prob(R=r\mid\hat{\btheta})=1/(n_{\mathrm{cal}}+1)$ for $r=1,\ldots,n_{\mathrm{cal}}+1$.

If $S_{\mathrm{new}}>S_{(k)}$, then at least $k$ calibration scores satisfy $S_i\le S_{(k)}<S_{\mathrm{new}}$, forcing $R\ge k+1$. By contrapositive, $\{R\le k\}\subseteq\{S_{\mathrm{new}}\le S_{(k)}\}$. Therefore
$\Prob(S_{\mathrm{new}}\le\hat\rho_\tau\mid\hat{\btheta}) \ge \Prob(R\le k\mid\hat{\btheta}) = k/(n_{\mathrm{cal}}+1) \ge \tau$.
Taking expectations over $\hat{\btheta}$ yields $\Prob(U_{\mathrm{new}}\in\cU_{\hat{\btheta},\hat\rho_\tau})\ge\tau$.
\end{proof}

\begin{proof}[Proof of Theorem~\ref{thm:consistency-eps}]
Fix $\varepsilon>0$ and $\btheta\in\Theta$. Write $S_i:=s_{\btheta}(U_i)$, and let $\hat F$, $F$ denote the empirical and population smoothed CDFs; $\hat\rho$, $\rho$ their $\tau$-quantiles.

\textbf{Step 1 (Uniform CDF convergence).}
Both $F$ and $\hat{F}$ are Lipschitz in $r$ with constant $L_F:=\|\varphi\|_\infty/\varepsilon$. Fix $\delta>0$ so that $\kappa:=\inf_{r\in[\rho-\delta,\rho+\delta]}\partial_rF(r)>0$. Covering $[\rho-\delta,\rho+\delta]$ with a finite grid of mesh $h=\eta/(4L_F)$, Lipschitz interpolation gives $\sup_{r\in[\rho-\delta,\rho+\delta]}|\hat{F}(r)-F(r)|\le\max_k|\hat{F}(r_k)-F(r_k)|+\eta/2$. The LLN at each grid point and a union bound yield $\sup_{r\in[\rho-\delta,\rho+\delta]}|\hat{F}(r)-F(r)|\xrightarrow{\Prob}0$.

\textbf{Step 2 (Quantile consistency).}
By the mean value theorem, $F(\rho+\eta)\ge\tau+\kappa\eta$ and $F(\rho-\eta)\le\tau-\kappa\eta$. On the event $\sup_{r\in[\rho-\delta,\rho+\delta]}|\hat{F}-F|\le\kappa\eta/2$, we get $\hat{F}(\rho+\eta)\ge\tau+\kappa\eta/2>\tau$ and $\hat{F}(\rho-\eta)\le\tau-\kappa\eta/2<\tau$, so monotonicity of $\hat{F}$ forces $|\hat\rho-\rho|\le\eta$. Hence $\hat\rho\xrightarrow{\Prob}\rho$.

\textbf{Step 3 (Quantile sensitivity consistency).}
Define $\omega_r(u):=\varphi((r-s_{\btheta}(u))/\varepsilon)$ and the empirical numerator/denominator $N_n(r):=n^{-1}\sum_i\omega_r(U_i)\nabla_{\btheta} s_{\btheta}(U_i)$, $D_n(r):=n^{-1}\sum_i\omega_r(U_i)$, with population limits $N(r)$, $D(r)$.

At fixed $r=\rho$: integrable domination $\|\omega_{\rho}(U)\nabla_{\btheta} s_{\btheta}(U)\|\le\|\varphi\|_\infty M(U)$ and the LLN give $N_n(\rho)\xrightarrow{\Prob}N(\rho)$, $D_n(\rho)\xrightarrow{\Prob}D(\rho)$.

For random $r=\hat{\rho}$: $|\omega_{\hat{\rho}}(U_i)-\omega_{\rho}(U_i)|\le\Omega(|\hat{\rho}-\rho|/\varepsilon)$ where $\Omega$ is the modulus of continuity of $\varphi$. Since $\hat{\rho}\to\rho$ in probability, $\Omega\to 0$. Combined with the LLN averages, $N_n(\hat{\rho})\xrightarrow{\Prob}N(\rho)$ and $D_n(\hat{\rho})\xrightarrow{\Prob}D(\rho)$. Since $D(\rho)=\varepsilon\partial_rF(\rho)>0$ by (A3), the continuous mapping theorem gives $N_n(\hat{\rho})/D_n(\hat{\rho})\xrightarrow{\Prob}\nabla_{\btheta}\rho_{P,\varepsilon}(\btheta)$.

\textbf{Step 4 (Combine).}
Define envelope terms \\ $H(r):=\sum_j\mu_j^*(\btheta,r)\nabla_{\btheta}\sigma_j$ and $G(r):=\sum_j\mu_j^*(\btheta,r)\sigma_{\cU_{\btheta,1}}(\bw_j)$. By (A4)--(A5), $H$ and $G$ are continuous at $\rho$, so $H(\hat{\rho})\xrightarrow{\Prob}H(\rho)$ and $G(\hat{\rho})\xrightarrow{\Prob}G(\rho)$. By Slutsky's theorem,
\[
\hat{\bg}_\varepsilon(\btheta) = H(\hat{\rho})+G(\hat{\rho})\cdot\frac{N_n(\hat{\rho})}{D_n(\hat{\rho})} \xrightarrow{\Prob} H(\rho)+G(\rho)\cdot\nabla_{\btheta}\rho_{P,\varepsilon}(\btheta) = \nabla_{\btheta} J_{P,\varepsilon}(\btheta).
\]
\end{proof}

\begin{corollary}[Coverage Preserved Under Tuned Training]\label{cor:coverage-preserved}
Let $\hat{\btheta}$ depend only on $\cD_{\text{train}} \cup \cD_{\text{tune}}$. Then conformal calibration on $\cD_{\text{cal}}$ yields $\Prob(U_{\text{new}} \in \cU_{\hat{\btheta}, \hat{\rho}_\tau}) \geq \tau$.
\end{corollary}

\section{Ellipsoidal Computation Details}
\label{app:ellipsoidal}

\begin{proposition}[Support Function for Ellipsoids]\label{prop:ellipsoid-support}
$\sigma_{\cU_{\bL,\rho}}(\bw) = \rho \norm{\bL^\top \bw}_2$.
\end{proposition}

\begin{proof}
$\sigma_{\cU_{\bL,\rho}}(\bw) = \sup_{\norm{\bm{v}}_2 \leq \rho} \ip{\bL^\top\bw}{\bm{v}} = \rho \norm{\bL^\top\bw}_2$.
\end{proof}

\begin{proposition}[Ellipsoidal Gradients]\label{prop:ellipsoid-grads}
For $\bw\neq 0$ with $\bL^\top\bw\neq 0$:
$\nabla_{\bL}\sigma_{\cU_{\bL,\rho}}(\bw) = \rho\bw\bw^\top\bL/\|\bL^\top\bw\|_2$.
For $\bu\neq 0$:
$\nabla_{\bL}s_{\bL}(\bu) = -\bL^{-\top}(\bL^{-1}\bu)(\bL^{-1}\bu)^\top/\|\bL^{-1}\bu\|_2$.
\end{proposition}

\begin{proof}
For the support function, let $y:=\bL^\top\bw$. Then $\mathrm{d}\sigma = \rho y^\top\mathrm{d}y/\|y\|_2 = \rho(\bL^\top\bw)^\top(\mathrm{d}\bL)^\top\bw/\|\bL^\top\bw\|_2$, which identifies the Frobenius gradient as $\rho\bw(\bL^\top\bw)^\top/\|\bL^\top\bw\|_2 = \rho\bw\bw^\top\bL/\|\bL^\top\bw\|_2$.

For the gauge, let $v:=\bL^{-1}\bu$. Then $\mathrm{d}v=-\bL^{-1}(\mathrm{d}\bL)v$ gives $\mathrm{d}s = -v^\top\bL^{-1}(\mathrm{d}\bL)v/\|v\|_2$, identifying the gradient as $-\bL^{-\top}vv^\top/\|v\|_2$.
\end{proof}

\section{Decoupled Zonal Benchmark and Additional Diagnostics}
\label{app:zonal-reserve}

The decoupled formulation removes the transfer constraints from the main text and retains only the zonal reserve requirements. Because the reserve margins depend explicitly on $(\bL,\rho)$, this benchmark isolates the core economic effect of learning the uncertainty-set geometry and provides a useful implementation check.

\subsection{Decoupled Benchmark Results}

Cost and reserves are evaluated at the conformally calibrated $\hat{\rho}_\tau$; for static methods cost is deterministic, for contextual it is averaged over test-set contexts. Reserve reports $\sum_z R_z^{\min}$.

\begin{table}[!t]
\caption{Appendix Benchmark: Decoupled Zonal Reserves}
\label{tab:decoupled}
\centering\footnotesize\setlength{\tabcolsep}{4pt}
\begin{tabular}{lcccc}
\toprule
\textbf{Method} & \textbf{Cost (\$/hr)} & \textbf{Reserve (MW)} & \textbf{Calibration Rate$^*$} & \textbf{Test Coverage [CI]$^\dagger$} \\
\midrule
Sample Covariance & 100{,}012 & 1{,}572 & 0.950 & 0.980 [.967, .990] \\
Independent & 99{,}944 & 1{,}563 & 0.950 & 0.986 [.975, .994] \\
Learned (Static) & 95{,}471 & 956 & 0.950 & 0.967 [.946, .985] \\
Learned (Contextual)$^\ddagger$ & 96{,}886 & 1{,}167 & 0.950 & 0.983 [.966, .995] \\
\bottomrule
\multicolumn{5}{l}{\footnotesize $^*$Calibration inclusion rate (empirical fraction of calibration points inside the set).} \\
\multicolumn{5}{l}{\footnotesize $^\dagger$Block bootstrap 95\% CI (block length 24\,hr, 10{,}000 replicates).} \\
\multicolumn{5}{l}{\footnotesize $^\ddagger$Contextual results for a single training seed.}
\end{tabular}
\end{table}

Table~\ref{tab:decoupled} shows that Learned (Static) reduces cost by 4.5\% (\$4{,}541/hr) relative to the Sample Covariance baseline and by 4.5\% (\$4{,}473/hr) relative to the Independent ablation, while reducing reserve procurement by about 39\% relative to Sample Covariance (1{,}572 to 956~MW). The gain comes almost entirely from reserve procurement (reserve component $\sum_i c_i^r r_i$: \$9{,}115/hr $\to$ \$4{,}805/hr, a 47\% decrease), while the energy component $\sum_i c_i^g g_i$ changes by less than \$163/hr. The Sample Covariance baseline performs comparably to the Independent ablation (\$100{,}012 vs.\ \$99{,}944), confirming that a statistically reasonable covariance estimate is not by itself sufficient for lower dispatch cost.

By construction, the split-conformal radius in Proposition~\ref{thm:conformal} yields calibration coverage at least $\tau$. Table~\ref{tab:decoupled} reports out-of-sample test coverage with 95\% CIs from a 24\,hr block bootstrap (10{,}000 replicates) to account for serial dependence; all methods exceed $\tau = 0.95$.

\subsection{Target-Level Sweep in the Decoupled Benchmark}

\begin{table}[!t]
\caption{Appendix Diagnostic: Cost--Coverage Tradeoff in the Decoupled Benchmark}
\label{tab:tau-sweep}
\centering
\begin{tabular}{lccccc}
\toprule
& \multicolumn{5}{c}{$\tau$} \\
\cmidrule(lr){2-6}
\textbf{Method} & 0.90 & 0.92 & 0.95 & 0.97 & 0.99 \\
\midrule
\multicolumn{6}{l}{\textit{Cost (\$/hr)}} \\
Sample Covariance & 99{,}186 & 99{,}491 & 100{,}012 & 100{,}489 & 101{,}501 \\
Independent & 99{,}119 & 99{,}481 & 99{,}944 & 100{,}530 & 101{,}523 \\
Learned (Static) & 94{,}951 & 95{,}134 & 95{,}471 & 95{,}981 & 96{,}809 \\
Learned (Contextual)$^\ddagger$ & 95{,}719 & 96{,}036 & 96{,}886 & 97{,}711 & 99{,}857 \\
\midrule
\multicolumn{6}{l}{\textit{Test Coverage}} \\
Sample Covariance & 0.949 & 0.964 & 0.980 & 0.987 & 0.995 \\
Independent & 0.964 & 0.975 & 0.986 & 0.993 & 0.998 \\
Learned (Static) & 0.931 & 0.947 & 0.967 & 0.982 & 0.989 \\
Learned (Contextual)$^\ddagger$ & 0.966 & 0.974 & 0.983 & 0.988 & 0.996 \\
\bottomrule
\multicolumn{6}{l}{\footnotesize $^\ddagger$Contextual results for a single training seed.}
\end{tabular}
\end{table}

Table~\ref{tab:tau-sweep} reports cost and test coverage across target levels $\tau \in \{0.90, 0.92, 0.95, 0.97, 0.99\}$, with each method's shape $\bL$ fixed at the $\tau = 0.95$ training point and only the conformal radius $\rho$ recalibrated at each $\tau$. Learned (Static) achieves lower cost than the Sample Covariance baseline at every $\tau$ while maintaining realized test coverage near or above the target, confirming that the economic advantage is not an artifact of evaluating at a looser coverage level. Independent again serves as a diagonal ablation. At comparable realized coverage (e.g., Independent at $\tau = 0.95$ gives 0.986 versus Learned (Static) at $\tau = 0.97$ gives 0.982), Learned (Static) remains \$3{,}963/hr cheaper.

\end{document}